\documentclass[a4paper, reqno, 11pt]{amsart}
\usepackage{amssymb,amsmath,amsthm,amsfonts,color} 

\usepackage[left=3.0cm,right=3.0cm,top=2.5cm,bottom=2.5cm]{geometry}

\usepackage{graphicx}
\usepackage{setspace}
\usepackage{comment}

\usepackage[colorlinks=true, pdfstartview=FitV, linkcolor=blue, 
            citecolor=blue, urlcolor=blue]{hyperref}

\numberwithin{equation}{section}

\newcommand{\hamiltonian}{{\bf h}}
\newcommand{\potential}{{\bf v}}
\newcommand{\dispersion}{\mathfrak{e}}

\newcommand{\p}{\partial}
\renewcommand{\epsilon}{\varepsilon}

\renewcommand{\hat}{\widehat}
\renewcommand{\Re}{{\rm Re}}

\newcommand{\C}{\mathbb{C}}

\newcommand{\R}{\mathbb{R}}
\newcommand{\T}{\mathbb{T}}

\newcommand{\Z}{\mathbb{Z}}

\newcommand{\cF}{\mathcal{F}}

\newcommand{\sign}{{\ensuremath{\mathrm{sign}}}}

\newcommand{\ess}{{\ensuremath{\mathrm{ess}}}}

\numberwithin{equation}{section}

\theoremstyle{plain}
\newtheorem{theorem}{Theorem}[section]

\newtheorem{lemma}[theorem]{Lemma}

\theoremstyle{definition}

\date{\today}

\begin{document}
\title[]{Asymptotics of eigenvalues of the zero-range perturbation of the discrete bilaplacian}

\author[Sh. Kholmatov] {Shokhrukh Yu. Kholmatov} 
\address[Shokhrukh Yu. Kholmatov]{University of Vienna\\ Oskar-Morgenstern Platz 1\\1090 Vienna  (Austria)}
\email[Sh. Kholmatov]{shokhrukh.kholmatov@univie.ac.at}

\author[M. Pardabaev] {Mardon  Pardabaev} 
\address[Mardon Pardabaev]{Samarkand State University\\ 
University boulevard 3\\ 140104 Samarkand (Uzbekistan)}
\email[M. Pardabaev]{p\_mardon75@mail.ru}



\begin{abstract}
We consider the family 
$$
\hat \hamiltonian_\mu:=\hat\varDelta\hat \varDelta - \mu \hat \potential,\qquad\mu\in\R,
$$ 
of discrete Schr\"odinger-type operators in one-dimensional lattice $\Z$, 
where $\hat \varDelta$ is the discrete Laplacian and $\hat\potential$ 
is of zero-range.  We prove that for any $\mu\ne0$  the discrete spectrum of $\hat \hamiltonian_\mu$ is  a singleton $\{e(\mu)\},$ and  $e(\mu)<0$ for $\mu>0$ and  $e(\mu)>4$ for $\mu<0.$ Moreover, we study the 
properties of $e(\mu)$ as a function of $\mu,$ in particular, we find the asymptotics of $e(\mu)$ as $\mu\searrow0$ and $\mu\nearrow0.$
\end{abstract}

\maketitle

\section{Introduction} 

In this paper we study the spectral properties of the family
$$
\hat \hamiltonian_\mu:=\hat\hamiltonian_0 - \mu\hat\potential,\qquad \mu\in\R,
$$
of self-adjoint bounded discrete Schr\"odinger-type operators in the Hilbert space $\ell^2(\Z)$ of square-summable complex-valued functions defined on the one-dimensional lattice $\Z.$  Here $\hat \hamiltonian_0$ is discrete bilaplacian, i.e. 
$$
\hat \hamiltonian_0:=\hat \varDelta\hat\varDelta,
$$
where
$$
\hat\varDelta\hat  f(x) =\hat f(x) - \frac{\hat f(x+1) + \hat f(x-1)}{2},\qquad \hat f\in \ell^2(\Z),
$$
is the discrete Laplacian,
and $\hat \potential$ is a rank-one operator 
$$
\hat \potential \hat f(x) = \begin{cases}
                            \hat  f(0), & x=0,\\
                             0, &x\ne0.
                            \end{cases}
$$ 
This model can be considered as the discrete Schr\"odinger operator on $\Z,$ associated to a system of one particle whose dispersion relation has a degenerate bottom. 

The spectral theory of discrete Schr\"odinger operators  with non-degenerate bottom in particular, with discrete Laplacian, have been extensively studied in recent years (see e.g. \cite{ALM:04:Puan,ALMM:2006,GSch:97,LKL:2012,LH:2011,Mat,Mog} and references therein) because of their applications in the theory of ultracold atoms in optical lattices \cite{JBC:1998,LSA:2012,Wall:2015,Wink:2006}. For these models the appearance of weakly coupled bound states has been sufficiently well-understood and  sufficient and necessary conditions for the existence of discrete spectrum in terms of the coupling constant have been established. This conditions naturally leads to the coupling constant threshold phenomenon \cite{KS:1980}: consider $-\Delta +\lambda V$ with $V$ short range  at a value $\lambda_0,$ where some eigenvalue $e(\lambda)$ is absorbed into the continuous spectrum as $\lambda\searrow \lambda_0,$ and conversely, for any $\epsilon>0,$ as $\lambda\nearrow \lambda_0+\epsilon$ the continuous spectrum gives birth to a new eigenvalue. This phenomenon and the absorption  rate of eigenvalues as $\lambda\to\lambda_0$ have been established for discrete Schr\"odinger operators with non-degenerate bottom, for example, in \cite{LKL:2012,LH:2011,LH:2012} and continuous Schr\"odinger operators, for example, in \cite{K:1977,KS:1980,S:1976}. 

In the case of Schr\"odinger operators with degenerate bottom some sufficient conditions for existence of bound states have been recently observed in \cite{HHV:2018}. However, to the best of our knowledge the results related to the absorption rate of eigenvalues have not been published yet. 

Recall that 
$$
\sigma(\hat \varDelta) = \sigma_\ess(\hat \varDelta) =[0,2], 
$$
thus, from the spectral theory  it follows that 
$$
\sigma(\hat \hamiltonian_0) = \sigma_\ess(\hat \hamiltonian_0) =[0,4],
$$
and hence, by the compactness of $\hat \potential $ and   Weyl's Theorem,
$$
\sigma_\ess(\hat \hamiltonian_\mu) = \sigma_\ess(\hat \hamiltonian_0) =[0,4]
$$
for any $\mu\in\R.$

Let   
$$
\cF:\ell^2(\Z)\to L^2(\T),\qquad \cF\hat f(p) = \frac{1}{\sqrt{2\pi }}\,\sum\limits_{x\in\Z} \hat f(x) e^{ixp} 
$$
be the standard Fourier transform with the inverse 
$$
\cF^{-1}:L^2(\T)\to \ell^2(\Z),\qquad \cF^{-1}f(x) = \frac{1}{\sqrt{2\pi}}\,\int_{\T} f(p) e^{-ixp}dp. 
$$
The operator $\hamiltonian_\mu = \cF\hat \hamiltonian_\mu\cF^{-1}$ is called the {\it momentum representation} of $\hat\hamiltonian_\mu.$ Note that 
\begin{equation}\label{hamilt_momentum}
\hamiltonian_\mu =\hamiltonian_0 - \mu\potential, 
\end{equation}
where $\hamiltonian_0:= \cF\hat \hamiltonian_0\cF^{-1}$ is the multiplication operator in $L^2(\T)$ by 
$$
\dispersion(q): =(1 - \cos q)^2,
$$
and $\potential: = \cF\hat \potential \cF^{-1}$  is the rank-one integral operator 
$$
\potential f(p) = \frac{1}{\sqrt{2\pi}} \int_\T f(q)dq.
$$

The main result of the paper is the following.

\begin{theorem}\label{teo:main_result}
For any $\mu\ne0$ the operator $\hamiltonian_\mu$ has a unique eigenvalue $e(\mu)$
outside the essential  spectrum with the associated eigenfunction 
\begin{equation}\label{eig_function}
f_\mu(p) = \frac{1}{\dispersion(q) - e(\mu)}. 
\end{equation}
Moreover:
\begin{itemize}
\item[(a)] $e(\mu)<0$ for any $\mu>0$ and $e(\mu)>4$ for any $\mu<0;$ 

\item[(b)] the function $\mu\in\R\setminus\{0\}\mapsto e(\mu)$ is 
real-analytic strictly decreasing, strictly convex in $(-\infty,0)$ and strictly concave in $(0,+\infty)$ with asymptotics
$$
\lim\limits_{\mu\to\pm\infty} \frac{e(\mu)}{\mu} = -1
$$
and 
\begin{equation*}
\lim\limits_{\mu\nearrow0} \frac{e(\mu) - 4}{\mu^2} =  \frac18,\qquad  
\lim\limits_{\mu\searrow0} \frac{e(\mu)}{\mu^{4/3}} = -\frac{1}{2^{4/3}};
\end{equation*}

\item[(c)] there exists $\gamma>0$ such that  
\begin{equation}\label{asymp1}
(e(\mu) - 4)^{1/2} = -\frac{\mu}{2\sqrt2} - \frac{\mu^2}{256} - \frac{27\mu^3}{9192\sqrt2} -\sum\limits_{n\ge4} C_n \mu^n  
\end{equation}
for any $\mu\in(-\gamma,0)$
and
\begin{equation}\label{asymp2}
(-e(\mu))^{1/4} =  
\frac{\mu^{1/3}}{2^{1/3}}+\frac{\mu}{24}-\frac{2^{1/3}\mu^{5/3}}{288} +\sum\limits_{n\ge3} D_n\mu^{(2n+1)/3}
\end{equation}
for any $\mu\in(0,\gamma),$ where $\{C_n\}$ and $\{D_n\}$ are some real coefficients.
\end{itemize}
\end{theorem}

\section{Proof of the main results}\label{sec:proof_mains}

We start with the following lemma in which we get an equation for eigenvalues of $\hamiltonian_\mu.$  

\begin{lemma}\label{lem:determinantga_kelish}
$z\in\C\setminus[0,4]$ is an eigenvalue of $\hamiltonian_\mu$ if and only if 
$\Delta(\mu;z)=0,$ where
$$
\Delta(\mu;z):=1 - \frac{\mu}{2\pi}\,\int_\T \frac{dq}{\dispersion(q) - z}.
$$
\end{lemma}

\begin{proof}
Suppose that $z\in \C\setminus[0,4]$  is an eigenvalue of $\hamiltonian_\mu$ with an associated eigenfunction $0\ne f\in L^2(\T).$
Then from $\hamiltonian_\mu f=zf$ it follows that 
\begin{equation} \label{f_p}
f(p) = \frac{C}{\dispersion(p) -z},
\end{equation}
where 
\begin{equation}\label{C}
C:=\frac{\mu}{2\pi} \int_\T f(q)dq.
\end{equation}
Inserting the representation \eqref{f_p} of $f$ in \eqref{C} we get
\begin{equation}\label{delta} 
C= \frac{C\mu}{2\pi} \int_\T \frac{dq}{\dispersion(q) - z}. 
\end{equation}
Since $C\ne0$ (otherwise $f=0$ by \eqref{f_p}), 
from \eqref{delta}  it follows that $\Delta(\mu;z)=0.$

Conversely, suppose that $\Delta(\mu;z)=0$ for some $z\in\C\setminus[0,4],$ 
and set 
$$
f(p):=\frac{1}{\dispersion(p) - z}.
$$
Then $f\in L^2(\T)\setminus\{0\}$ and
$$
(\hamiltonian_\mu - z)f(p) = 1 - \frac{\mu}{2\pi}\int_\T \frac{dq}{\dispersion(q) - z} = \Delta(\mu;z) =0,
$$
i.e.  $z$ is an eigenvalue of $\hamiltonian_\mu$ with associated eigenvector $f.$
\end{proof}

Notice that 
$$
z\in\C\setminus[0,4] \mapsto \Delta(\mu;z) 
$$ 
is analytic for any $\mu\in\R$ and and 
$$
\frac{\p }{\p z}\,\Delta(\mu;z) = -\mu \int_\T\frac{dq}{(\dispersion(q) - z)^2}.
$$
Thus, for any $\mu>0$ resp. for any $\mu<0,$ the function $\Delta(\mu;\cdot)$ is strictly decreasing resp. strictly increasing in $\R\setminus [0,4].$ Moreover,
\begin{equation}\label{Delta_xosdda}
\lim\limits_{z\to\pm\infty} \Delta(\mu;z) =1.  
\end{equation}

\begin{lemma}\label{lem:det_hisob}
For any  $z\in\R\setminus[0,4],$
\begin{equation}\label{Iz_def}
\frac1{\sqrt{2}\pi}\,\int_\T\frac{dq}{\dispersion(q) - z} =  - \frac{\sign(z)}{
\sqrt[4]{z^3(z-4)} } \,
\left( 1 +\sqrt{ \frac{z}{z-4} } \right)^{1/2},
\end{equation}
where $\sign(z)$ is the sign of $z,$ i.e. equal to $1$ if $z>0,$ to $0$ if $z=0$ and $-1$ if $z<0.$
\end{lemma}

\begin{proof}
We establish \eqref{Iz_def} for $z<0;$ one can prove it for $z>4$ using similar arguments.
For simplicity, set
$z=-\alpha^4$ for $\alpha>0.$ Using the Euler formula $e^{iq} = \cos q+i\sin q,$
we rewrite the integral as 
$$
I(z):=\int_\T\frac{dq}{\dispersion(q) - z} =\int_{\T} \frac{dq}{\big(1 - 
\frac{e^{iq} +e^{-iq}}{2}\big)^2 +\alpha^4} = 
\int_{\T} \frac{4e^{2iq}dq}{\big(e^{2iq}- 2e^{iq}+1\big)^2 +4\alpha^4e^{2iq}} 
$$
so that 
$$
I(z)
=
-4i\int_\T \frac{e^{iq}de^{iq}}{ (e^{iq}- 1)^4 +4\alpha^4e^{2iq}}.
$$
Changing variables  $\xi=e^{iq}$ we see that 
$$
I(z)
=
-4i\int_{|\xi|=1} \frac{\xi d\xi}{(\xi- 1)^4 +4\alpha^4\xi^2}.
$$
Since the only zeros $\xi_0:= (1 - \frac{\alpha}{A}) + i(A\alpha-\alpha^2)$ and $\overline{\xi_0}$  of the fourth-order polynomial 
$(\xi- 1)^4 +4\alpha^4\xi^2$ with real coefficients belong to the ball $\{|\xi|<1\},$ where 
\begin{equation}\label{def_A}
A:=\sqrt{\frac{\sqrt{4+\alpha^4} +\alpha^2}{2}}, 
\end{equation}
by the  Residue Theorem for analytic functions,
$$
I(z)= 8\pi \Big(\frac{\xi_0}{4(\xi_0- 1)^3 + 8\alpha^4\xi_0} + \frac{\overline{\xi_0}}{4(\overline{\xi_0}- 1)^3 + 8\alpha^4\overline{\xi_0}} \Big)=
4\pi\,\Re\,\frac{\xi_0}{(\xi_0- 1)^3 + 2\alpha^4\xi_0}. 
$$
Since $(\xi_0- 1)^4 +4\alpha^4\xi_0^2=0,$ one has  $(\xi_0- 1)^3 = -\frac{4\xi_0^2}{\xi_0-1},$ thus,
$$
I(z)= -\frac{2\pi}{\alpha^4}\,\Re\,\frac{\xi_0-1}{\xi_0+1}. 
$$
Using the definition \eqref{def_A} of $A,$ by the direct computation one finds 
$$
\frac{\xi_0-1}{\xi_0+1} = -\frac{A^2 + i}{A\sqrt{4+\alpha^4}}\,\alpha,  
$$
therefore, 
$$
I(z) = \frac{2A\pi }{\alpha^3\,\sqrt{4+\alpha^4}},
$$
and \eqref{Iz_def} follows.
\end{proof}

\begin{proof}[Proof of Theorem \ref{teo:main_result}]

Rewriting \eqref{Iz_def} as
\begin{equation}\label{Iz_def1}
\frac1{\sqrt{2}\pi}\,\int_\T\frac{dq}{\dispersion(q) - z} =  - \frac{\sign(z)}{
\sqrt[4]{|z|^3}\,\sqrt{|z-4|}  } \,
\sqrt{\sqrt{|z-4|} +\sqrt{|z|}},
\end{equation}
we observe that
\begin{equation}\label{ahsbda}
\lim\limits_{z\nearrow 0} \int_\T\frac{dq}{\dispersion(q) - z} =+\infty,\qquad  \lim\limits_{z\searrow 4} \int_\T\frac{dq}{\dispersion(q) - z} =-\infty. 
\end{equation}
Since 
\begin{equation}\label{Delta_xos}
\Delta\ge1   
\end{equation}
for $(\mu;z)\in [0,+\infty)\times(4,+\infty)$ and for $(\mu;z)\in (-\infty, 0]\times(-\infty,0],$
by
\eqref{ahsbda}, \eqref{Delta_xosdda} and the strict monotonicity and the analyticity of $z\in\R\setminus[0,4]\mapsto \Delta(\mu;z),$ for any $\mu\ne0$ there exists a unique $e(\mu)\in \R\setminus[0,4]$ such that 
$
\Delta(\mu;e(\mu)) = 0.                                                                                                                                                                                  $  

(a) By Lemma \ref{lem:determinantga_kelish}, $e(\mu)$ is the unique eigenvalue of $\hamiltonian_\mu$ with the associated eigenfunction \eqref{eig_function}.
Moreover, from \eqref{Delta_xos} it follows that $e(\mu)<0$ for $\mu>0$ and $e(\mu)>4$ for $\mu<0.$

(b) By the Implicit Function Theorem, the function $\mu\in \R\setminus\{0\}\mapsto e(\mu)$ is real-analytic. Moreover, computing the derivatives of the implicit function $e(\mu)$ we find:
\begin{equation}\label{e_mu_der}
e'(\mu) = - \frac{1}{\mu}\,\int_\T \frac{dq}{\dispersion(q) - e(\mu)}\,\Big( \int_\T \frac{dq}{(\dispersion(q) - e(\mu))^2}\Big)^{-1},\qquad \mu\ne0, 
\end{equation}
thus, using $\mu(\dispersion(q) - e(\mu))>0$ we get $e'(\mu)<0,$ i.e. $e(\cdot)$ is strictly decreasing in $\R\setminus\{0\}.$
Differentiating \eqref{e_mu_der} once more and using $\mu\int_\T\frac{dq}{\dispersion(q) - e(\mu)}=1$ we get
$$
e''(\mu) = \frac{2e'(\mu)}{\mu}\,\left(1 - \mu e'(\mu)\,\int_\T \frac{dq}{(\dispersion(q) - e(\mu))^3}\,\left( \int_\T \frac{dq}{(\dispersion(q) - e(\mu))^2}\right)^{-1}\right).
$$
Therefore, $e''(\mu)>0$ (i.e. $e(\cdot)$ is strictly convex) for $\mu<0$ and $e''(\mu)<0$ (i.e. $e(\cdot)$ is strictly concave) for $\mu>0.$

By \eqref{Iz_def}, $e(\mu)$ solves 
\begin{equation}\label{mu_nega}
\sqrt{2} = \frac{-\mu}{\sqrt[4]{e(\mu)^3}\,\sqrt{e(\mu)-4}  } \,
\sqrt{\sqrt{e(\mu)-4} +\sqrt{e(\mu)}}\qquad\text{for $\mu<0$} 
\end{equation}
and
\begin{equation}\label{mu_posi}
\sqrt{2} = \frac{\mu}{\sqrt[4]{-e(\mu)^3}\,\sqrt{4-e(\mu)}  } \,
\sqrt{\sqrt{4- e(\mu) } +\sqrt{ -e(\mu)}}\qquad\text{for $\mu>0.$}
\end{equation}
The strict monotonicity of $e(\cdot)$ and \eqref{mu_nega} and \eqref{mu_posi} imply that 
\begin{equation}\label{mina_senga_asympto}
\lim\limits_{\mu\to-\infty} e(\mu) = +\infty,\qquad 
\lim\limits_{\mu\nearrow 0}e(\mu) = 4,\qquad 
\lim\limits_{\mu\searrow 0}e(\mu) = 0,\qquad \lim\limits_{\mu\to+\infty} e(\mu) = -\infty,  
\end{equation}
hence, $e(\cdot)$ has a jump at $\mu=0.$ 
In particular, from \eqref{mu_nega} and \eqref{mina_senga_asympto} we obtain 
$$
\lim\limits_{\mu\to-\infty} \frac{e(\mu)}{-\mu} = 1 \qquad \text{and}\qquad 
\lim\limits_{\mu\nearrow0} \frac{\sqrt{4-e(\mu)}}{-\mu} =\frac{1}{\sqrt{2}}\lim\limits_{\mu\nearrow0} \frac{1}{\sqrt{e(\mu)}} = \frac{1}{2\sqrt2}.
$$
Analogously, by \eqref{mu_posi} and \eqref{mina_senga_asympto},
$$
\lim\limits_{\mu\to+\infty} \frac{-e(\mu)}{\mu} = 1 \qquad \text{and}\qquad 
\lim\limits_{\mu\searrow0} \frac{\sqrt[4]{-e(\mu)^3}}{\mu} =   \frac{1}{\sqrt{2}}\lim\limits_{\mu\searrow0} \frac{1}{\sqrt[4]{4-e(\mu)}} = \frac{1}{2}.
$$

Now we establish \eqref{asymp1}. Setting $\alpha:=\alpha(\mu)=\sqrt{e(\mu) - 4}$ we rewrite \eqref{mu_nega}  as 
\begin{equation}\label{eq_for_alpha1}
\alpha = \frac{-\mu}{2^{3/2}(1+\frac{\alpha}{4})^{3/4}} \,\Big(\Big(1 + \frac{\alpha^2}{4}\Big)^{1/2} + \frac{\alpha}{2}\Big)^{1/2}. 
\end{equation}
By (b), $\mu\in(-\infty,0)\mapsto \alpha(\mu)$ is strictly decreasing. In view of \eqref{mina_senga_asympto} there exists  $\gamma^1>0$ such that $\alpha\in(0,1)$ for any $\mu\in (-\gamma^1,0).$ Using  
\begin{equation}\label{ildiz_x}
(1+x)^{1/2} = 1 + \frac x2 + \sum\limits_{n\ge2} \frac{(-1)^{n-1}}{2^nn!}\,\prod_{j=0}^{n-2}(1+2j)\,x^n 
\end{equation}
for $|x|<1,$  one has
\begin{align*}
\Big(\Big(1 + \frac{\alpha^2}{4}\Big)^{1/2} + &\frac{\alpha}{2}\Big)^{1/2} =   
\Big(1 + \frac{\alpha}{2} +\frac{\alpha^2}{8} +\sum\limits_{n\ge2} \frac{(-1)^{n-1}}{8^nn!}\prod\limits_{j=0}^{n-2}(1+2j) \alpha^{2n} \Big)^{1/2}\\
= & 1 + \frac{\alpha}{4} +\frac{\alpha^2}{16} +\sum\limits_{m\ge2} \frac{(-1)^{m-1}}{2\cdot 8^mm!}\prod\limits_{j=0}^{m-2}(1+2j) \alpha^{2m}\\
&+ \sum\limits_{n\ge2} \frac{(-1)^{n-1}}{ 2^nn!}  \prod\limits_{j=0}^{n-2}(1+2j)   
\Big(\frac{\alpha}{2} +\frac{\alpha^2}{8} +\sum\limits_{m\ge2} \frac{(-1)^{m-1}}{8^mm!}\prod\limits_{j=0}^{m-2}(1+2j) \alpha^{2m}\Big)^n. 
\end{align*}
Thus, using 
$$
(1 + x)^{-3/4} = 1+ \sum\limits_{n\ge1} \frac{(-1)^n}{4^nn!}\prod\limits_{j=0}^{n-1} (3+4j) x^n
$$
for $|x|<1$ we get 
\begin{equation}\label{aadfa}
\frac{1}{(1+\frac{\alpha}{4})^{3/4}} \,\Big(\Big(1 + \frac{\alpha^2}{4}\Big)^{1/2} + \frac{\alpha}{2}\Big)^{1/2} = 1-\frac{\alpha}{16}+\frac{13\alpha^2}{512} +\sum\limits_{n\ge3} c_n\alpha^n,
\end{equation}
where $\{c_n\}$ are some real coefficients. Using \eqref{aadfa} we rewrite \eqref{eq_for_alpha1} as 
\begin{equation}\label{sadadada}
\alpha = -\frac{\mu}{2^{3/2}}\,\Big(1-\frac{\alpha}{16}+\frac{13\alpha^2}{512} +\sum\limits_{n\ge3} c_n\alpha^n\Big)
\end{equation}
for $\mu\in(-\gamma^1,0).$
Setting $\alpha= -\frac{\mu}{2^{3/2}}(1+u)$ we can represent \eqref{sadadada} as 
$$
F(u,\mu) =0,
$$
where 
$$
F(u,\mu):=u- \frac{\mu(1+u)}{16\cdot 2^{3/2}} - \frac{13\mu^2(1+u)^2}{2048} - \sum\limits_{n\ge3} \frac{c_n}{2^{3n/2}}\,\mu^n(1+u)^n
$$
is real-analytic in a neighborhood of $(u,\mu)=(0,0),$ 
$$
F(0,0)=0,\qquad F_u(0,0) =1.
$$
Hence, by the Analytic Implicit Function Theorem, there exist $\gamma>0$ and a unique function $u=u(\mu)$ given by a series $u(\mu) = \sum\limits_{n\ge1}a_n\mu^n$ with real coefficients $\{a_n\}$  and absolutely convergent for $|\mu|<\gamma$ such that $F(u(\mu),\mu)\equiv0$ for any $|\mu|<\gamma.$
Inserting $\alpha= -\frac{\mu}{4^{3/4}}(1+u)$ and   $u(\mu) = \sum\limits_{n\ge1}a_n\mu^n$ in 
\eqref{sadadada} one finds inductively $a_1=\frac{1}{32\sqrt2},$ $a_2=\frac{27}{4096}$ and so on.
This implies \eqref{asymp1}.

Now we prove \eqref{asymp2}.
Set $\mu:=\lambda^3$ and $e(\mu) = -\alpha^4,$ where $\alpha:=\alpha(\lambda)>0$ is the strictly increasing function of $\lambda>0.$ Then \eqref{mu_posi} is rewritten as
\begin{equation}\label{eq_for_alpha}
\sqrt2 = \frac{\lambda^3}{\alpha^3\sqrt{4+\alpha^4}}\,\sqrt{\sqrt{4+\alpha^4} +\alpha^2}. 
\end{equation}
We solve this equation with respect to $\alpha.$ 
To this aim first we rewrite it the right hand side of \eqref{eq_for_alpha} as an absolutely convergent series of $\alpha,$ and then from the Implicit Function Theorem in analytical case we deduce that 
$\alpha$ has a convergent power series in $\lambda$ and the coefficients of the series will be found inductively from the series  in $\alpha$. 

By the strict decrease of $\alpha(\cdot)$ and asymptotics \eqref{mina_senga_asympto} there exists a unique $\gamma_1>0$ such that $|\alpha(\lambda)|<1$ for any $\lambda\in(0,\gamma_1).$ 
Recalling for $|x|<1$ that 
$$
(1+x)^{-1/4} = 1 + \sum\limits_{n\ge1} \frac{(-1)^n}{4^nn!}\,\prod_{j=0}^{n-1}(1+4j)\,x^n,
$$
and
$$
(1+x)^{-1/2} = 1 + \sum\limits_{n\ge1} \frac{(-1)^n}{2^nn!}\,\prod_{j=0}^{n-1}(1+2j)\,x^n,
$$
as well as \eqref{ildiz_x}
we get 
$$
\frac{\sqrt2}{(4+\alpha^4)^{1/4}} = 1+
\sum\limits_{n\ge1} \frac{(-1)^n}{16^nn!}\,\prod_{j=0}^{n-1}(1+4j)\,\alpha^{4n}.
$$
and 
$$
\frac{\alpha^2}{\sqrt{4+\alpha^4} }= \frac{\alpha^2}{2}\,
\Big(1+ \frac{\alpha^4}{4}\Big)^{-1/2} = \frac{\alpha^2}{2} + 
\sum\limits_{n\ge1} \frac{(-1)^n}{2\cdot 8^nn!}\,\prod_{j=0}^{n-1}(1+2j)\,\alpha^{4n+2}
$$
so that 
\begin{align*}
\sqrt{1+\frac{\alpha^2}{\sqrt{4+\alpha^4} }} =& 1+
\frac{\alpha^2}{4} + 
\sum\limits_{n\ge1} \frac{(-1)^n}{4\cdot 8^nn!}\,\prod_{j=0}^{n-1}(1+2j)\,\alpha^{4n+2}   \\ %
+ & \sum\limits_{n\ge 2} \frac{(-1)^{n-1}}{2^nn!}\,\prod_{j=0}^{n-2}(1+2j)\,
\Big(\frac{\alpha^2}{2} + 
\sum\limits_{m\ge1} \frac{(-1)^m}{2\cdot 8^mm!}\,
\prod_{j=0}^{m-1}(1+2j)\,\alpha^{4m+2}\Big)^n.
\end{align*}
This implies 
$$
\frac{\sqrt2}{(4+\alpha^4)^{1/4}}\,
\sqrt{1+\sqrt{\frac{\alpha^4}{4+\alpha^4} }} = 
1 + \frac{\alpha^2}{4} - \frac{\alpha^4}{32} +\sum\limits_{n\ge3} C_n\alpha^{2n},
$$
where $C_n,$ $n=3,4,\ldots,$ are real coefficients.
Hence for $\lambda\in(0,\gamma_1)$ the equation \eqref{eq_for_alpha} is represented as 
\begin{equation}\label{dadfwe}
2\alpha^3 = \lambda^3\Big(1 + \frac{\alpha^2}{4} - \frac{\alpha^4}{32} +\sum\limits_{n\ge3}
C_n\alpha^{2n}\Big). 
\end{equation}
Let $\alpha = \lambda(2^{-1/3} + u)$ we rewrite  \eqref{dadfwe} as 
\begin{equation}\label{dasgg}
2^{1/3}u+2^{2/3}u^2 +\frac23\,u^3 = \frac{\lambda^2(2^{-1/3} + u)^2}{12} 
- \frac{\lambda^4(2^{-1/3} + u)^4}{96} + \sum\limits_{n\ge3} 
\frac{C_n}{3}\, \lambda^{2n}(2^{-1/3} + u)^{2n}.
\end{equation}
For $\mu:=\lambda^2,$ \eqref{dasgg} is represented as 
$$
F(u,\mu)=0,
$$
where 
$$
F(u,\mu):=2^{1/3}u+2^{2/3}u^2 +\frac23\,u^3 - \frac{\mu(2^{-1/3} + u)^2}{12} + \frac{\mu^2(2^{-1/3} + u)^4}{96} - \sum\limits_{n\ge3} \frac{C_n}{3}\, \mu^{n}(2^{-1/3} + u)^{2n}
$$
is analytic in a neighborhood of $(u,\mu)=(0,0),$ and satisfies 
$$
F(0,0) = 0\qquad  \text{and} \qquad F_u(0,0)=2^{1/3}>0.
$$ 
Hence, by the Implicit Function Theorem 
in analytical case there exists $\gamma>0$ and a real-analytic function 
$u=u(\mu)$ given by the absolutely convergent series $u(\mu)=\sum\limits_{n\ge1}a_n\mu^n$ with $\{a_n\}\subset\R$ for $|\mu|<\gamma$  and $F(u(\mu),\mu)\equiv0$ for any $\mu\in(-\gamma,\gamma).$ Inserting  $u(\mu)=\sum\limits_{n\ge1}a_n\mu^n$ in \eqref{dasgg} we find the coefficients $a_k$ inductively: 
$$
a_1=\frac{1}{24},\qquad a_2=-\frac{2^{1/3}}{288}
$$ 
and so on. 
Therefore,
$$
\alpha(\lambda)=2^{1/3}\lambda+\frac{1}{24}\,\lambda^3-\frac{2^{1/3}}{288}\,\lambda^5+\sum\limits_{n\ge3} a_n\lambda^{2n+1}.
$$
Now the definitions of $\alpha$ and $\lambda$ imply \eqref{asymp2}.
\end{proof}

\section*{Acknowledgments}
The first author acknowledges support from the Austrian Science Fund (FWF) project M~2571-N32.

\end{document}